\documentclass[12pt, a4paper,leqno]{article}

\usepackage[english]{babel}
\usepackage[utf8]{inputenc}
\usepackage[T1]{fontenc}

\usepackage{graphicx}

\usepackage{paralist}

\usepackage{mathtools}
\usepackage{amssymb}
\usepackage{amsthm}
\usepackage{accents}
\usepackage{mathrsfs}
\usepackage{dsfont}

\newtheorem{thm}{Theorem}[section]
\newtheorem{lem}[thm]{Lemma}

\newtheorem{rem}[thm]{Remark}

\theoremstyle{definition}

\newcommand{\norm}[2]{||#1||_{#2}}

\newcommand{\RR}{\mathbb{R}}

\usepackage{verbatim}

\numberwithin{equation}{section}

\title{Blow-up of solutions to relaxed compressible Navier-Stokes equations in divergence form}
\author{Johannes Bärlin}

\begin{document}
\maketitle

\section{Compressible Navier-Stokes equations with relaxation in divergence form}
We consider a hyperbolic system of balance laws with relaxation given by
\begin{align}
\label{br0}
\left\lbrace \begin{aligned} \rho_t + (\rho u)_x & = 0\\
(\rho u)_t + (\rho u^2 + p(\rho))_x & = S_x\\
\tau((\rho S)_t + (\rho u S)_x) - u_x & = -S
\end{aligned} \right.
\end{align}
where $\tau > 0$ and the pressure $p$ satisfies the constitutive relation
\begin{align}
p(\rho) = \rho^\gamma
\end{align}
for $\gamma > 1$. System \eqref{br0} is to be solved on $[0,T) \times \RR$ for some $T \in (0,\infty]$ with $(\rho,u,S)$ taking values in the state space $(0,\infty) \times \RR \times \RR$. It is equivalent to
\begin{align}
\label{br1} \rho_t + (\rho u)_x & = 0\\
\label{br2} \rho u_t + \rho u u_x + p(\rho)_x & = S_x\\
\label{br3} \tau\rho(S_t + u S_x) - u_x & = -S.
\end{align}

System \eqref{br0} is a delayed version of the Navier-Stokes equations with a Maxwell-type law as a relaxation equation. A divergence-form variant of systems recently considered by Hu et al. \cite{hw18,hrw22}, it is a prototypical version of a model proposed by Ruggeri \cite{ru83} as discussed by Freistühler in \cite{hf22}. In the latter the question is raised whether, like the systems considered in \cite{hw18,hrw22}, also divergence-form relaxation systems such as \eqref{br0} possess a dichotomy in the sense that for small perturbations of a homogeneous reference state the Cauchy problem for \eqref{br0} has unique global solutions which time-asymptotically decay to the reference state, while for (some) large data there is a blow-up of solutions, i.e. physically reasonable solutions can only exist on a finite time interval. This note serves to confirm the second part of the mentioned dichotomy.

In proving the claimed blow-up we closely follow the procedure given by Hu, Racke and Wang \cite{hrw22}. In fact we show that system \eqref{br0} has all the necessary features in order to apply the identical steps as in \cite{hrw22} from a certain point on. Our system is some what similar to the one considered in \cite{hw18}, and in this note one may find many analogies to \cite{hw18} including a dissipative-entropy equation and a finite speed of propagation result for the system \eqref{br0}.

The general strategy has its origins in two remarkable papers of T. Sideris. In \cite{s84} he establishes a finite speed of propagation result for symmetric hyperbolic conservation laws which he uses in \cite{s85} to show the \grqq formation of singularities in three-dimensional compressible fluids\grqq. The blow-up occurs for a certain averaged quantity involving $\rho u$. The strategy in \cite{s85} is by now well-established and for instance used in \cite{hw18} and \cite{hrw22}.

The system \eqref{br0} belongs to a general class of quasi-linear equations considered by Godunov \cite{g61} and Boillat \cite{bo74}. Equations in this Godunov-Boillat class have the property of possessing a dissipative entropy equation. In this note we will use, and shortly derive, the follwing mathematical entropy equation
\begin{align}
\label{entropydissipationequality}
0 = & \left\lbrack \frac{p(\rho) - p(\overline{\rho})- \gamma(\rho-\overline{\rho})}{\gamma - 1} + \frac{\tau \rho S^2}{2} + \frac{\rho u^2}{2} \right\rbrack_t\\
\nonumber & + \left\lbrack \frac{\gamma}{\gamma -1}u p(\rho)+ \frac{\gamma}{\gamma-1} \rho u + \frac{\rho u^3}{2} + \frac{\tau \rho u S^2}{2} - uS \right\rbrack_x\\
\nonumber & + S^2
\end{align}
where $\overline{\rho} > 0$ is some reference density. Assume we are given a solution $(\rho,u,S) \in C^1(\Omega)$ on some open subset $\Omega \subset (0,\infty) \times \RR)$. To obtain \eqref{entropydissipationequality} multiply \eqref{br1} by $p^\prime (\rho)$ to get
\begin{align*}
0 = p(\rho)_t + \gamma u_x p(\rho) + u p(\rho)_x = p(\rho)_t + [\gamma u p(\rho)]_x + (1-\gamma) u p(\rho)_x.
\end{align*}
In this way we find
\begin{align}
\label{upx1}
u p(\rho)_x = \left\lbrack \frac{1}{\gamma -1} p(\rho)\right\rbrack _t + \left\lbrack \frac{\gamma}{\gamma -1} u p(\rho)\right\rbrack _x.
\end{align}
From \eqref{br1} and \eqref{br2} we deduce
\begin{align}
\label{upx2}
-up(\rho)_x + u S_x = \rho u u_t + \rho u^2 u_x + \frac{u^2}{2}(\rho_t + (\rho u)_x) = \left\lbrack \frac{\rho u^2}{2} \right\rbrack_t + \left\lbrack \frac{\rho u^3}{2} \right\rbrack_x
\end{align}
and similarly from \eqref{br3}
\begin{align}
\label{Sequation}
- S^2 + u_x S = \tau (\rho S_t S + \rho u S_x S) = \left\lbrack \frac{\tau \rho S^2}{2} \right \rbrack_t + \left\lbrack \frac{\tau\rho u S^2}{2} \right \rbrack_x .
\end{align}
Combining \eqref{upx1}-\eqref{Sequation} gives the entropy dissipation equality \eqref{entropydissipationequality} where we have added an affine linear expression in the time-derivative term  and compensated for this appropriately in the space-derivative term.
\section{Blow-up of solutions for some large data}

We begin with two properties of the Cauchy problem for system \eqref{br1}-\eqref{br3}. On the one hand we have existence of physically reasonable solutions of said Cauchy problem on a maximal time interval  for suitable inital values since \eqref{br1}-\eqref{br3} may be written as a first-order symmetrizable hyperbolic system of quasi-linear differential equations. On the other hand there is a finite speed of propagation result in the spirit of \cite{s84}.
\begin{lem}
\label{lemmaexistence}
Let $\overline{\rho} > 0$. Let $(\rho_0,u_0,S_0):\RR \rightarrow \RR^3$ satisfy $(\rho_0-\overline{\rho},u_0,S_0) \in H^2(\RR)$ and $\rho_0(x) > 0$ for all $x \in \RR$. Then there exists a maximal $T \in (0,\infty]$ such that there is a unique solution $(\rho,u,S)$ of \eqref{br1}-\eqref{br3} satisfying
\begin{align*}
(\rho-\overline{\rho},u,S) \in C^0([0,T),H^2(\RR)) \cap C^1([0,T),H^1(\RR)),
\end{align*}
and
\begin{align*}
(\rho(0),u(0),S(0)) = (\rho_0,u_0,S_0),
\end{align*}
and for all $(t,x) \in [0,T) \times \RR$
\begin{align*}
\rho(t,x) > 0.
\end{align*}
\end{lem}
\begin{proof}
See Chapter 2 in \cite{ma84}.
\end{proof}
\begin{lem}
\label{lemmafinitespeed}
Let $T, R > 0$ and $\overline{\rho} > 0$. Set
\begin{align}
\label{sigma}
\sigma := \sqrt{p^\prime(\overline{\rho}) + \frac{1}{\tau \overline{\rho}^2}} > 0.
\end{align}
Suppose $(\rho,u,S) \in C^1([0,T) \times \RR)$ solves \eqref{br1}-\eqref{br3} and satisfies
\begin{itemize}
\item $(\rho,u,S) \in C^1_b([0,t]\times \RR)$ for all $t \in [0,T)$;
\item $\textnormal{supp}\, (\rho(0,\cdot) - \overline{\rho}, u(0),S(0)) \subset (-R,R)$.
\end{itemize}
Then for all $t \in [0,T)$ and $x \in D_t:= \lbrace x \in \RR: \; |x| \geq R + \sigma t \rbrace$ it holds
\begin{align}
(\rho(t,x) - \overline{\rho},u(t,x),S(t,x)) = (0,0,0).
\end{align}
\end{lem}
\begin{proof}
Write $U = (\rho,u,S)$ and define
\begin{align*}
A_0(U) = \begin{pmatrix}
1 & 0& 0\\ 0 & \rho & 0\\ 0& 0 & \tau \rho
\end{pmatrix}, \; A_1(U) = \begin{pmatrix}
u & \rho& 0\\ p^\prime(\rho) & u & -1\\ 0& -1 & \tau \rho u
\end{pmatrix},\; B = \begin{pmatrix}
0 & 0& 0\\ 0 & 0 & 0\\ 0& 0 & 1
\end{pmatrix}.
\end{align*}
Set $\overline{U} := (\overline{\rho},0,0)$ and consider for $V := U -\overline{U}$ the system
\begin{align*}
V_t + A_0(\overline{U})^{-1} A_1(\overline{U}) V_x + A_0(\overline{U})^{-1} B V = G(\overline{U},U,U_t,U_x)
\end{align*}
with
\begin{align*}
G(\overline{U},U,U_t,U_x) = A_0(\overline{U})^{-1} \lbrack (A_0(\overline{U}) - A_0(U)) U_t + (A_1(\overline{U}) - A_1(U)) U_x \rbrack.
\end{align*}
The eigenvalues of $A_0(\overline{U})^{-1} A_1(\overline{U})$ are
\begin{align*}
\lambda_1(\overline{U},\tau) = -\sqrt{p^\prime(\overline{\rho}) + \tfrac{1}{\tau \overline{\rho}^2}}, \; \lambda_2(\overline{U},\tau) = 0, \; \lambda_3(\overline{U},\tau) = \sqrt{p^\prime(\overline{\rho}) + \tfrac{1}{\tau \overline{\rho}^2}}.
\end{align*}
Set $\sigma := \lambda_3$ and proceed like in \cite{hw18} which follows the idea of \cite{s84} to show the claim of the lemma.
\end{proof}
The data for which solutions of \eqref{br1}-\eqref{br3} experience a blow-up is constructed from the function in the following lemma taken from \cite{hr14} (see \cite{hrw22}, too):
\begin{lem}
Let $L > 0$ and $M > 2$. Then the function
\begin{align*}
u_{L,M}(x) := \left\lbrace \begin{aligned}
0, && x \in (-\infty,-M]\\
\tfrac{L}{2} \cos(\pi(x+M)) - \tfrac{L}{2}, && x \in (-M,-M+1]\\
-L, && x \in (-M+1,-1]\\
L \cos(\pi(x+M)), && x \in (-1,1]\\
L, && x \in (1,M-1]\\
\tfrac{L}{2} \cos(\pi(x+M)) + \tfrac{L}{2}, && x \in (M-1,M]\\
0, && x \in (M,\infty)
\end{aligned} \right.
\end{align*}
is an element of $H^2(\RR)\cap C^1(\RR)$ with
\begin{align}
\label{uL2}
\norm{u_{L,M}}{L^2}^2 \leq 2 L^2M.
\end{align}
\end{lem}
Now we state and prove the theorem on blow-up of solutions to \eqref{br1}-\eqref{br3} for some large data.
\begin{thm}
Let $R > 0, \overline{\rho} = 1$ and let $(\rho_0,S_0) \in C^1(\RR)$. Suppose $\rho_0(x) > 0$ for all $x \in \RR$ and $\textnormal{supp}\; (\rho_0-1,S_0) \subset (-R,R)$. Assume
\begin{align}
\label{initialmass}
\int_\RR \rho_0(x) -1 dx \geq 0.
\end{align}
and set $\sigma := \sqrt{\gamma + \tau^{-1}}$. Choose $\tilde{\sigma},L\ > 0$ such that
\begin{align}
\label{largespeed}
\tilde{\sigma}^2 = \max\lbrace \sigma^2, (8 \max \rho_0)^{-1} \rbrace
\end{align}
and
\begin{align}
\label{largedata}
\frac{L}{2} \min \rho_0 > \max \lbrace \sqrt{8 \max \rho_0}, 16 \tilde{\sigma} \max \rho_0 \rbrace.
\end{align}
Define
\begin{align}
\label{H_0}
H_0 := \int_\RR \frac{p(\rho_0(x))-1 - \gamma (\rho_0(x) -1)}{\gamma - 1} + \frac{\tau \rho_0(x) S^2_0(x)}{2} dx
\end{align}
and choose $M \geq \max \lbrace 4,R \rbrace$ such that
\begin{align}
\label{largesupport}
H_0 + \max \rho_0 L^2 M \leq 2 \tilde{\sigma}M^2 \max \rho_0.
\end{align}
Let $(\rho,u,S)$ be the solution of \eqref{br1}-\eqref{br3} with initial data $(\rho_0,u_{L,M},S_0)$ and maximal time of existence $T > 0$ as given by Lemma \ref{lemmaexistence}. Then $T$ is finite and the critical averaged quantity initially satisfies
\begin{align}
\label{initialcriticaltreshold}
\int_\RR x \cdot \rho_0(x) u_0(x) dx > \max \lbrace 16 \tilde{\sigma} M^2 \max \rho_0, M^2\sqrt{8 \max \rho_0}\rbrace.
\end{align}
\begin{rem}
Before we enter the proof let us comment on where the above choices of constants enter in the proof. The relations \eqref{largespeed} and \eqref{largedata} are needed to establish two a-priori estimates which appear in the deduction of the central Riccati-type differential inequality. Then again \eqref{largedata} together with \eqref{largesupport} serves to find that the quantity on the right of \eqref{initialcriticaltreshold} plays the role of a critical treshold for the Riccati-type inequality thereby allowing the conclusion of a finite maximal time of existence of the considered solution of \eqref{br1}-\eqref{br3}.
\end{rem}
\begin{proof}
The proof is very close to section 3 of \cite{hrw22}, also with respect to notation. In fact after establishing the Riccati-type estimate the claim follows like in \cite{hrw22}. The general strategy follows \cite{s85}. First define for $t \in [0,T)$ the averaged quantities
\begin{align}
\label{mass}
m(t) := \int_\RR \rho(t,x) - 1 dx
\end{align}
and
\begin{align}
\label{momentum}
F(t) := \int_\RR x \cdot \rho(t,x)u(t,x) dx.
\end{align}
$m$ and $F$ are well-defined and continuous functions because of the compact supports of $u(t,\cdot)$ and $\rho(t,\cdot)-1$ by Lemma \ref{lemmafinitespeed}.

By conservation of density \eqref{br1}, the compact support of $\rho(t,\cdot)-1$ and by \eqref{initialmass} we have for $t \in [0,T)$
\begin{align}
\label{conservationm}
m(t) = m(0) \geq 0.
\end{align}
Setting
\begin{align*}
B_t := \lbrace x \in \RR:\; |x| < M + \tilde{\sigma}t \rbrace \subset D_t^c
\end{align*}
and using Jensen's inequality (see p.846 [WH18]) implies
\begin{align}
\label{pmonotony}
\int_{B_t} p(\rho(t,x)) dx \geq \int_{B_t} p(\overline{\rho}) dx.
\end{align}
By Lemma \ref{lemmafinitespeed}, partial integration and \eqref{br1} and \eqref{br2} we have
\begin{align*}
F^\prime(t) = \int_\RR (\rho u^2)(t,x) dx + \int_\RR p(\rho(t,x))-p(\overline{\rho}) dx - \int_\RR S(t,x).
\end{align*}
Apply estimate \eqref{pmonotony} and Young's inequality to conclude
\begin{align}
\label{Fprime}
F^\prime(t) \geq \int_\RR (\rho u^2)(t,x) dx - \frac{1}{2} \int_\RR S^2(t,x) dx - (M + \tilde{\sigma} t).
\end{align}
Hölder's inequality and \eqref{conservationm} yield
\begin{align}
\label{F2}
F^2(t) & \leq \int_{B_t} x^2 \rho(t,x) dx \int_{B_t} (\rho u^2)(t,x)dx\\
\nonumber& \leq 2(M+\tilde{\sigma}t)^3 \max \rho_0 \int_{B_t} (\rho u^2)(t,x)dx.
\end{align}
Let
\begin{align*}
c_2 & := \frac{\tilde{\sigma}}{M},\\
c_3 & := \frac{1}{2 \max \rho_0 M^3},\\
c_1 & := \frac{2 c_2}{c_3},
\end{align*}
and assume for all $t \in [0,T)$ the a-priori estimates
\begin{align}
\label{apriori1}
F(t) \geq c_1 > 0
\end{align}
and
\begin{align}
\label{apriori2}
M+\tilde{\sigma}t = M(1+c_2t) \leq \frac{c_3}{2(1+c_2t)^3} F^2(t).
\end{align}
Combining the estimate \eqref{Fprime} for $F^\prime$ with the estimate \eqref{F2} for $F^2$ and using the a-priori estimates \eqref{apriori1} and \eqref{apriori2} one finds the Riccati-type inequality
\begin{align}
\label{Riccati}
\frac{F^\prime(t)}{F^2(t)} \geq \frac{c_3}{2(1+c_2t)^3} - \frac{1}{2c_1} \int_\RR S^2(t,x)dx.
\end{align}
One uses the entropy equality \eqref{entropydissipationequality} to treat the $S^2$-term above:
\begin{align}
\label{S2}
\int_\RR S^2(t,x) dx & \leq \int_\RR \frac{p(\rho_0) -1 -\gamma(\rho_0-1)}{\gamma -1} + \frac{\rho_0 u_{L,M}^2}{2} + \frac{\tau \rho_0 S^2}{2}dx\\
\nonumber& \leq H_0 + \frac{\max \rho_0}{2} \norm{u_{L,M}}{L^2}^2.
\end{align}
\begin{rem}
By Taylor's Theorem one has for $\rho > 0$
\begin{align*}
p(\rho) = 1 + \gamma (\rho-1) + \frac{p^{\prime \prime}}{2}(\xi)(\rho-1)^2
\end{align*}
for some $\xi > 0$. Note $p^{\prime \prime}(\rho) > 0$ for $\rho > 0$. These facts were used for the estimate in \eqref{S2} (see also p.829 \cite{hw18}).
\end{rem}
Now we are in the exact same situation as in \cite{hrw22} section 3, and all further steps repeat theirs. Collecting the constants once more in
\begin{align*}
c_4 := \frac{H_0}{2c_1^2} \; \textnormal{and} \; c_5 := \frac{\max \rho_0}{4 c_1^2}
\end{align*}
and noting
\begin{align*}
c_4 + c_5 \norm{u_{L,M}}{L^2}^2 \leq \frac{c_3^2}{8c_2^2} (H_0 + \max \rho_0 L^2 M) \leq \frac{c_3}{8c_2}
\end{align*}
by  \eqref{uL2}, the definition of $c_1$ and \eqref{largesupport} we finally find by integrating \eqref{Riccati} that for all $t \in [0,T)$
\begin{align}
\label{F01}
F(0)^{-1} & \geq F(0)^{-1} - F(t)^{-1}\\
\nonumber & \geq - \frac{c_3}{4c_2(1+c_2t)^2} + \frac{c_3}{4c_2} - c_4 -c_5 \norm{u_{L_M}}{L^2}^2\\
\nonumber& \geq - \frac{c_3}{4c_2(1+c_2t)^2} + \frac{c_3}{8c_2}.
\end{align}
But by definition of $u_{L,M}$ (note its symmetry $u_{L,M}(x)=-u_{L,M}(-x)$), $M \geq 4$ and by \eqref{largedata} it holds
\begin{align}
\label{F02}
F(0) > \frac{L}{2} \min \rho_0 M^2 \geq 16 \frac{\tilde{\sigma}}{M} \max \rho_0 M^3 = \frac{8 c_2}{c_3}
\end{align}
which implies $T < \infty$ since otherwise there exists a $t>0$ such that \eqref{F01} violates \eqref{F02}.

It remains to verify the a-priori estimates. We have 
\begin{align*}
F(0) \geq 2 c_1
\end{align*}
by \eqref{F02}. If for some $T^\ast \in (0,T]$ we have $F \geq c_1$ on $[0,T^\ast)$ then we get from \eqref{F01}
\begin{align*}
F(t) \geq \frac{4c_2(1+c_2t)^2}{c_3} \geq 2 c_1.
\end{align*}
for all $t \in [0,T^\ast)$. Since $F$ is continuous these two facts imply $F \geq c_1$ on $[0,T)$ (see p.10 \cite{hr14}).

We show
\begin{align}
\label{apriori3}
M+\tilde{\sigma}t = M(1+c_2t) \leq \frac{c_3}{4(1+c_2t)^3} F^2(t)
\end{align}
which implies \eqref{apriori2}. Proceeding in a similar fashion like before we have
\begin{align*}
F^2(0) \geq \frac{L^2 M^4 \min \rho_0^2}{2} \geq 4 M^4 \max \rho_0 = \frac{4M}{c_3} = 2 \frac{2M}{c_3}
\end{align*}
by \eqref{F02} and \eqref{largedata}. Hence the estimate \eqref{apriori3} holds for $t = 0$. Again if for some $T^\ast \in (0,T]$ the estimate \eqref{apriori2} holds for all $t \in [0,T^\ast)$ then by the Riccati-type inequality \eqref{F01} and assumption \eqref{largespeed} we find
\begin{align*}
F^2(t) \geq \frac{16 c_2^2}{c_3^2}(1+c_2t)^4 \geq \frac{4M}{c_3} (1+c_2t)^4
\end{align*}
closing the a-priori estimate \eqref{apriori2} by a continuity argument.
\end{proof}
\end{thm}
\bibliographystyle{alpha}

\end{document}